%% file: orbcplxcob_25_8.tex
\documentclass[11pt]{amsart}
\usepackage{amssymb,amscd}
\usepackage{graphicx}

\usepackage[colorlinks=true, urlcolor=blue,bookmarks=true,bookmarksopen=true, citecolor=blue,hypertex]{hyperref}

\numberwithin{equation}{section}

\newtheorem{theorem}{Theorem}[section]
\newtheorem{defn}[theorem]{Definition}

\newtheorem{corollary}[theorem]{Corollary}

\newtheorem{lemma}[theorem]{Lemma}
\newtheorem{prop}[theorem]{Proposition}
\newtheorem{remark}[theorem]{Remark}
\newtheorem{ques}[theorem]{Question}

\theoremstyle{definition}
\newtheorem{obser}[theorem]{Observation}
\newtheorem{example}[theorem]{Example}

\def \begineq{\begin{equation}}
\def \endeq{\end{equation}}

\def \bb{\mathbb}

\def \CC{{\bb{C}}}
\def \CP{{\bb{CP}}}

\def \QQ{{\bb{Q}}}
\def \RR{{\bb{R}}}
\def \TT{{\bb{T}}}

\def \ZZ{{\bb{Z}}}

\def \({\left(}
\def \){\right)}
\def \<{\langle}
\def \>{\rangle}
\def \bar{\overline}
\def \deg{\mathrm{deg}}

\def \tensor{\otimes}

\def \Aut{{\rm Aut}}

\def \qed{\hfill $\square$ \vspace{0.03in}}

\begin{document}

\title{Complex Cobordism of quasitoric orbifolds}

\author[S. Sarkar]{Soumen Sarkar}
\address{\it{Department of Mathematics and Statistics, University of Regina,
Regina S4S 6A7, Canada}}

\email{soumensarkar20@gmail.com}

\subjclass[2010]{55N22, 57R90}

\keywords{quasitoric orbifolds, stable and almost complex structure, complex cobordism}

\abstract We construct orbifolds with quasitoric boundary and show that they have stable almost complex structure.
We show that a quasitoric orbifold is complex cobordant to finite disjoint copies of complex orbifold
projective spaces.
Finally some computations in the complex cobordism ring for manifolds are given. 
\endabstract
\maketitle

\section{Introduction}\label{intro}
Cobordism was explicitly introduced by Lev Pontryagin in his seminal paper \cite{Pon1}.
 In \cite{Tho} Thom showed that the cobordism groups could be computed
by results of homotopy theory using the Thom complex construction.
Later, Atiyah \cite{At1} showed that complex cobordism is a generalized
cohomology theory. In Section 1 of \cite{Qui}, Quillen discussed 
geometric interpretation of complex cobordism rings. Following his definition, we define
the complex orbifold (co)bordism groups and rings for the category of stable almost complex orbifolds.
It seems that complex cobordism for orbifolds did not appear in the literature until now. 
However oriented cobordism of orbifolds is studied in \cite{Dru1} and \cite{Dru2}.

In the pioneering paper \cite{DJ}, Davis and Januskiewicz introduced the topological counterpart of nonsingular projective
toric varieties. They called this class of manifolds toric manifolds. Since ``toric manifold'' is used in algebraic 
geometry for ``nonsingular toric variety'', Buchstaber and Panov \cite{BP} introduced the term ``quasitoric
manifold'' instead. Quasitoric orbifolds are generalization of quasitoric manifolds and they are
studied in \cite{PS}.
An orbifold with quasitoric boundary
is an orbifold with boundary where the boundary is a disjoint union of some quasitoric
orbifolds.

In this article we study the complex cobordism of quasitoric orbifolds. The article
is organized as follows.
In Section \ref{orb}, we recall the definition of stable complex structure on an orbifold.
In Section \ref{qo}, we recall the definition of quasitoric orbifold, omniorientation
on a quasitoric orbifold and equivariant classification of quasitoric orbifolds. Also
we show that a quasitoric orbifold over a simplex is equivariantly homeomorphic to a
complex orbifold projective space, see Lemma \ref{cops}. In Section \ref{def}, we construct
oriented orbifolds with quasitoric boundary. In Section \ref{ccob}, first we show that the
orbifolds with quasitoric boundary which are constructed in Section \ref{def} have stable
complex structure, see Theorem \ref{comcob}. Then we show that
a quasitoric orbifold is complex cobordant to finite disjoint copies of complex orbifold
projective spaces, see Theorem \ref{qbd}. This process produces explicit complex
cobordism relations among quasitoric orbifolds. We show that the set of all complex
orbifold cobordism classes of complex orbifold projective spaces is not linearly 
independent, see Observation \ref{obs}. Note that this is in contrast to manifold case.
As a particular case when the orbifold
singularity is trivial, we get explicit complex cobordism relations among quasitoric manifolds.
At the end of Section \ref{ccob}, we give some sufficient conditions to the famous problem of Hirzebruch
which asks when a complex cobordism class in the complex cobordism ring $\Omega^U$ for manifolds
may contain a connected nonsingular algebraic variety.
In \cite{W}, Andrew Wilfong gives some necessary condition of this problem up to dimension 8.
We also compute the Chern numbers of a quasitoric manifold over a simplex, see Example \ref{cncpn}.


\section{orbifolds}\label{orb}
Orbifolds were introduced by Satake \cite{Sat}, who called them $V$-manifolds.
An orbifold is a singular space that is locally look like a quotient of an
open subset of Euclidean space by an action of a finite group. The readers are referred
to the Section 1.1 in \cite{ALR} for the definition and basic facts concerning effective orbifolds.
Also they may see \cite{MM} for an excellent exposition of the foundation of the theory of the
reduced differentiable orbifolds.

Similarly as the definition of manifold with boundary, we can talk about orbifold with
boundary, see Definition 1.3 in \cite{Dru1}. In this article Druschel studies the
orientation on orbifolds in Section 1.

Many concepts in orbifold theory are defined in the context of groupoid, see \cite{ALR} to enjoy
this approach. For example, Section 2.3 in \cite{ALR} talks about orbifold vector bundle in the language
of groupoid. Most relevant example of the orbibundle of an orbifold is its tangent bundle. An explicit
description of the tangent bundle of an effective orbifold is given in Section 1.3 of \cite{ALR}.

\begin{defn}
Let $Y$ be a smooth orbifold with the tangent bundle $(TY, p_Y, Y)$ where $p_Y \colon TY \to Y$ is the projection map.
\begin{enumerate}
\item An almost complex structure on $Y$ is an endomorphism $J \colon TY \to TY$ such that $J^2 = -Id$.

\item A stable almost complex (or stable complex) structure on $Y$ is an endomorphism
\begin{equation}
J \colon TY \oplus (Y \times \RR^k) \to TY \oplus (Y \times \RR^k)
\end{equation}
such that $J^2 = -Id$ for some positive integer $k$.

\end{enumerate}
\end{defn}

\section{Quasitoric orbifolds}\label{qo}
In this Section we review the definition of quasitoric orbifold following \cite{PS}. 
We also discuss several results on quasitoric orbifolds.
An $n$-dimensional {\em simple polytope} in $\RR^n$ is a convex polytope where exactly $n$ bounding hyperplanes meet at 
each vertex.
A $facet$ is a codimension one face of a convex polytope. 
If $P$ is a convex polytope then we denote the set of all facets of $P$ by $\mathcal{F}(P)$.
Let $\TT^n = (\ZZ^n \tensor_{\ZZ} \RR)/\ZZ^n$ and $\TT_M = (M \tensor_{\ZZ} \RR)/M$ for a free $\ZZ$-module $M$. 

\begin{defn}
A $2n$-dimensional quasitoric orbifold $ X$ is a smooth orbifold with a $\TT^n$-action,
such that the orbit space is (diffeomorphic as manifold with corners to) an $n$-dimensional simple polytope $P$.
Denote the projection map from $X$ to $P$ by $\pi \colon X \to P$. Furthermore every point $x \in X$ has
\begin{itemize}
\item[A1)] a $\TT^n$-invariant neighborhood $V$,
\item[A2)] an associated free  $\ZZ$-module $M$ of rank $n$ with an
  isomorphism $\theta \colon \TT_M \to U(1)^n$ and an injective module homomorphism $\iota \colon M \to \ZZ^n$
  which induces a surjective covering homomorphism $\iota_M \colon \TT_M \to \TT^n $,
\item[A3)] an orbifold chart $(W, G, \eta)$ over $V$ where $W$ is $\theta$-equivariantly diffeomorphic to an
    open set in $\CC^n$,  $G = {\ker} \iota_M $ and $\eta \colon W \to V$ is an equivariant map i.e.
   $\eta(t \cdot y)= \iota_M (t) \cdot \eta(y)$ inducing a homeomorphism between $W/G$ and $V$.
\end{itemize}
\end{defn}

Note that this definition is a generalization of the axiomatic definition of a quasitoric manifold,
see Section 1 in \cite{DJ}. Let $P$ be an $n$-dimensional simple polytope and 
$\mathcal{F}(P) = \{P_1, \ldots, P_s\}$.
\begin{defn}\label{dich}
A function $\xi \colon \mathcal{F}(P) \to \ZZ^n$ is called a di-characteristic function if the 
vectors $\xi(P_{j_1}), \ldots, \xi(P_{j_l})$ are linearly independent over $\ZZ$ whenever
the intersection of the facets $P_{j_1}, \ldots, P_{j_l}$ is nonempty.

The vector $ \xi_j = \xi(P_{j})$ is called the di-characteristic vector corresponding to the facet $P_j$
and the pair $(P, \xi)$ is called a characteristic model on $P$.
\end{defn}
\begin{remark}
 If the set $\{\xi(P_{j_1}), \ldots, \xi(P_{j_l})\}$ is a part of a basis of $\ZZ^n$
 over $\ZZ$ whenever the intersection of the facets $P_{j_1}, \ldots, P_{j_l}$ is nonempty, then the map $\xi$
 is called characteristic function on $P$, see Section 1 of \cite{DJ}.
\end{remark}

In Subsection 2.1 of \cite{PS}, the authors construct a quasitoric orbifold from the characteristic
model $(P, \xi)$. Also given a quasitoric orbifold we can associate a characteristic model to it up
to choice of signs of di-characteristic vectors.

\begin{example}\label{td}
 Let $S^{2n+1}=\{(z_0, \ldots, z_n) \in \CC^{n+1} \colon |z_0|^2 + \cdots + |z_n|^2=1\}$ and $a_0, \ldots, a_n$
 be coprime positive integers. Then a weighted action of the circle $S^1$ on $S^{2n+1}$ is given by
$$\alpha \cdot (z_0, \ldots, z_n) = (\alpha^{a_0} z_0, \ldots, \alpha^n z_n) \quad \mbox{for} ~\alpha \in S^1.$$
The orbit space $\mathbb{WP}(a_0, \ldots, a_n) =S^{2n+1}/S^1$ is called a weighted projective space.
 Since $a_0, \ldots, a_n$ are coprime integers, the vector
$\mathfrak{a}=(a_0, \ldots, a_n) \in \ZZ^{n+1}$ determines a circle subgroup $S^1_{\mathfrak{a}}$ of $\TT^{n+1}$.
Then the natural $\TT^{n+1}$-action
on $S^{2n+1}$ induces an action of $\TT^n \cong \TT^{n+1}/S^1_{\mathfrak{a}}$ on $\mathbb{WP}(a_0, \ldots, a_n)$.
With respect to this $\TT^n$-action, $\mathbb{WP}(a_0, \ldots, a_n)$ is a quasitoric orbifold over an $n$-simplex.
For an integer $a > 1$, the space $\mathbb{WP}(1, a)$ is called the teardrop.
The characteristic model for the teardrop is given by $([0,1], \xi)$ where $\xi(\{0\})=-1$ and $\xi(\{1\})=a$ 
(possibly up to sign). \qed
\end{example}

\begin{defn} 
Let $\delta \colon \TT^n \to \TT^n$ be an automorphism. Two quasitoric orbifolds $X_1$ and $X_2$ over the same
polytope $P$ are called $\delta$-equivariantly homeomorphic if there is a homeomorphism $ f \colon X_1 \to X_2$
such that $f(t \cdot x) = \delta(t)\cdot f(x)$ for all $(t, x ) \in \TT^n \times X_1$.
\end{defn}
The automorphism $\delta$ induces an automorphism $\delta_{\ast}$ of $\ZZ^n$.
For the automorphism $\delta$,
two characteristic models $(P, \xi)$ and $(P, \eta)$ are called $\delta$-{\em equivalent} if there is a diffeomorphism
$\psi: P \to P$ (as manifold with corners) such that $\eta(\psi(F))= \pm \delta_{\ast}(\xi(F))$ for all $F \in \mathcal{F}(P)$.
If $\delta$ is identity, then $(P, \xi)$ and $(P, \eta)$ are called {\em equivalent}. 
The following proposition is a classification result which can be found 
in \cite{BP} for quasitoric manifolds (Proposition 5.14) and in \cite{PS} for quasitoric orbifolds (Lemma 2.2).
\begin{prop}\label{probi}
For an automorphism $\delta$ of the torus $\TT^n$, there is a bijection between $\delta$-equivariant homeomorphism
classes of quasitoric orbifolds over $P$ and $\delta$-equivalent classes of characteristic models on $P$.
\end{prop}
Suppose $\delta$ is the identity automorphism of $\TT^n$. Proposition \ref{probi} implies that two quasitoric
orbifolds over $P$ are equivariantly homeomorphic if and only if their di-characteristic models are equivalent.

Let $X$ be a quasitoric orbifold with the orbit map $\pi \colon X \to P$.
There are important closed $\TT^n$-invariant suborbifolds of $X$ which are corresponding to the faces of $P$.
If $F$ is a codimension $k$ face of $P$, then
define $X(F) = \pi^{-1}(F)$. The space $X(F)$ with subspace topology is a quasitoric orbifold of dimension $2n-2k$.
If $F$ is a facet of $P$ then $X(F)$ is called a {\it characteristic suborbifold} of $X$. Note that
a choice of orientation on $\TT^n $ and $ P$ give an orientation on the orbifold $X$.
\begin{defn}
 An {\it omniorientation} of a quasitoric orbifold $X$ is a choice of orientation
for $X$ as well as an orientation for each characteristic suborbifold of $X$.
\end{defn}
Clearly the isotropy group of a characteristic suborbifold is a circle subgroup of $\TT^n$.
So there is a natural $S^1$-action on the normal bundle (possibly an orbibundle) of that characteristic suborbifold.
Thus the normal bundle has a complex structure and consequently an orientation. Whenever the
sign of the characteristic vector of a facet is reverse, we get the opposite orientation on
the normal bundle. An orientation on the normal bundle together with an orientation on $X$
induces an orientation on the characteristic suborbifold. So a di-characteristic function
determines a natural omniorientation. We call this omniorientation the {\it characteristic omniorientation}.

A toric variety $X_{\Sigma}$ associated to the simplicial fan $\Sigma$ is called a toric orbifold.
The space $X_{\Sigma}$ is compact if and only if $\Sigma$ is complete. More about toric varieties
can be found in \cite{CK}. 

\begin{defn}
Let $\Sigma$ be a complete simplicial fan in $\RR^n$ with $n+1$ many one-dimensional cones. The associated
toric orbifold $X_{\Sigma}$ is called a $2n$-dimensional complex orbifold projective space. 
\end{defn}
\begin{lemma}\label{cops}
 Let $X$ be a quasitoric orbifold over $\bigtriangleup^n$. Then $X$ is equivariantly homeomorphic to a
 $2n$-dimensional complex orbifold projective space.
\end{lemma}
\begin{proof}
Let $X$ be a quasitoric orbifold over $\bigtriangleup^n$ and $\mathcal{F}(\bigtriangleup^n) =
\{F_0, \ldots, F_n\}$. Let $$\xi : \mathcal{F}(\bigtriangleup^n) \to \ZZ^n$$ be the associated
di-characteristic function. Suppose $\xi_i = \xi(F_i)$ for $i=0, \ldots, n$.
So $\{\xi_0, \ldots, \hat{\xi_i}, \ldots, \xi_n\}$ is a linearly independent set in $\ZZ^n$ for $i=0, \ldots, n$
where $\hat{}$ represents the omission of the corresponding entry. Let
$$\xi_0 = a_1\xi_1 + \cdots + a_n \xi_n$$ for some $a_1, \ldots, a_n \in \QQ$. Then $a_i \neq 0$
for all $i \in \{1, \ldots, n\}$. Suppose $a_{i_1}, \ldots, a_{i_l} \in \QQ_{>0}$. Define
$\eta \colon : \mathcal{F}(\bigtriangleup^n) \to \ZZ^n $ by
\begin{equation}
\eta(F_j) = \left\{ \begin{array}{ll} -\xi_{j} & \mbox{if} ~ j \in \{i_1, \ldots, i_l\}\\
\xi_j & \mbox{if} ~ j \in \{0, \ldots, n\} - \{i_1, \ldots, i_l\}.
\end{array} \right.
\end{equation}
Let $\eta_j = \eta(F_j)$ for $j=0, \ldots, n$ and 
$$
b_j = \left\{ \begin{array}{ll} -a_{j} & \mbox{if} ~ j \in \{i_1, \ldots, i_l\}\\
a_j & \mbox{if} ~ j \in \{0, \ldots, n\} - \{i_1, \ldots, i_l\}.
\end{array} \right.
$$
So $ b_j < 0$ for $j=1, \ldots, n$ and $\eta_0= b_1 \eta_1 + \cdots + b_n \eta_n$.
Therefore $\eta_0, \ldots, \eta_n$ are the one-dimensional cones of a complete simplicial fan $\Sigma$ in $\RR^n$. Let
$X_{\Sigma}$ be the associated toric orbifold. So $X_{\Sigma}$ is a complex orbifold projective space. With respect
to the compact $n$-torus action on $X_{\Sigma}$, it is a quasitoric orbifold where the corresponding characteristic
model is $(\bigtriangleup^n, \eta)$. Therefore by Proposition \ref{probi}, $X$ and $X_{\Sigma}$ are equivariantly
homeomorphic. 
\end{proof}
Fake weighted projective space is a holomorphic generalization of weighted projective space, see \cite{Ka}.
A $2n$-dimensional fake weighted projective space is defined by a complete simplicial fan generated by
$n+1$ many primitive vectors in $\ZZ^n$. So a fake weighted projective space
is a complex orbifold projective space. Since the primitive vectors in $\ZZ$ are $-1 $ and $ 1$, the teardrop
$\mathbb{WP}(1, a)$ is not a fake weighted projective space if $a > 1$. But $\mathbb{WP}(1, a)$ is a
complex orbifold projective space.

\section{Construction of orbifolds with quasitoric boundary}\label{def}
In this section we construct some oriented orbifolds with quasitoric boundary. 
This construction is a generalization of the construction of manifolds with quasitoric boundary of Section
$4$ in \cite{Sar2}.
\begin{defn}
An $(n+1)$-dimensional simple polytope $Q$ in $\RR^{n+1}$ is said to be polytope with exceptional facets
$\{Q_1, \ldots, Q_k\}$ if $Q_i \cap Q_j$ is empty for $i \neq j$ with $1 \leq i, j \leq k$ and $V(Q) = \cup_{i=1}^k V(Q_i)$.
We denote $\{Q \backslash Q_1, \ldots, Q_k\}$ for a simple polytope with exceptional facets.
\end{defn}
Let $ \mathcal{F}(Q) = \{F_1, \ldots, F_m\} \cup \{Q_1, \ldots, Q_k\}$ be the facets of $Q$ where $\{Q_1,
\ldots, Q_k\}$ are the exceptional facets.
\begin{defn}\label{isofun}
A function $ \lambda \colon \{F_1, \ldots, F_m\} \to \ZZ^{n}$ is called an isotropy function on
$\{Q \backslash Q_1, \ldots, Q_k\}$ if
the vectors $ \lambda(F_{i_1}), \ldots, \lambda(F_{i_{q}})$ are linearly independent in $\ZZ^{n}$
whenever the intersection of the facets $ F_{i_1}, \ldots,$ $ F_{i_q}$ is nonempty. The vector $\lambda_i =
\lambda(F_i)$ is called an isotropy vector assigned to the facet $F_i$ for $i=1, \ldots, m$.
\end{defn}
We define the isotropy function on some polytopes with exceptional facets in Example \ref{egchar2}.
\begin{remark}\label{re1}
Since $\ZZ^n$ is not the union of finitely many proper submodules over $\ZZ$, we can define the isotropy function
on any polytope with exceptional facets.
\end{remark}

We adhere the notations of Definition \ref{isofun}.
Let $ F $ be a codimension $l$ face of $ Q $ with $ 0< l \leq n+1$. If $F$ is a face of $Q_i$ for some
$i \in \{1, \ldots, k\}$, then $ F $ is the intersection of a unique collection of $ l $ facets $F_{i_1},
\ldots, F_{i_{l-1}}, Q_i$ of $ Q $. Otherwise, $ F $ is the intersection of a unique collection of $ l $
facets $F_{i_1}, \ldots, F_{i_{l}} \in \{F_1, \ldots, F_m\}$ of $Q$. Let
\begin{equation}
M(F) = \left\{ \begin{array}{ll} \< \{\lambda_{i_j} \colon j = 1, \ldots, l-1\}\> \subseteq \ZZ^{n} & \mbox{if} ~ F =
F_{i_1} \cap \cdots \cap F_{i_{l-1}} \cap Q_i\\
\< \{\lambda_{i_j} \colon j = 1, \ldots, l\}\> \subseteq \ZZ^{n} & \mbox{if} ~ F = F_{i_1} \cap \cdots \cap F_{i_{l}}
\end{array} \right.
\end{equation}
where $\<\{\alpha_i : i=1, \ldots, s\}\>$ denotes the submodule
 generated by the vectors $\{\alpha_{i} \colon i = 1, \ldots, s\}$ of $\ZZ^n$.
So
\begin{equation}
\TT_{M(F)} = (M(F) \tensor_{\ZZ} \RR)/M(F)
\end{equation}
is a compact torus of dimension $l-1$ or $l$ depending on the situation of the face $F$.
Adopt the convention that $\TT_{M(Q)} = 1= \TT_{M(Q_i)}$ for $i=1,
\ldots, k$.  The inclusion $M(F) \hookrightarrow \ZZ^n$ induces a natural homomorphism $$f_F \colon \TT_{M(F)} \to \TT^{n}$$
for any face $F$ of $Q$. Denote the image of $f_F$ by $Im(f_F)$. Define an equivalence relation
$\sim_b$ on the product $ \TT^{n} \times Q$ as follows,
\begin{equation}\label{equilam}
(t, x) \sim_b (u, y) \quad \mbox{if and only if} \quad x = y ~ \mbox{ and} ~ tu^{-1} \in Im(f_F)
\end{equation}
where $ F $ is the unique face of $ Q $ containing $ y $ in its relative interior. We denote the quotient
space $ (\TT^{n} \times Q)/ \sim_b $ by $ W(Q, \lambda)$ and the equivalence class of $(t, x)$ by $[t, x]^{\sim_b}$.
The space $W(Q, \lambda)$ is a $\TT^{n}$-space where the action is induced by the group
operation in $\TT^{n}$. Let $$\mathfrak{q} \colon W(Q, \lambda) \to Q$$ be the projection map defined by
$\mathfrak{q}([t,x]^{\sim_b}) = x$.
We consider the standard orientation of $\mathbb{T}^{n}$ and the orientation of $Q$ induced from the ambient space $\RR^{n+1}$.

\begin{lemma}\label{orbbd}
The space $ W(Q, \lambda) $ is a $(2n+1)$-dimensional oriented orbifold with boundary and the boundary is a
disjoint union of $2n$-dimensional quasitoric orbifolds.
\end{lemma}
\begin{proof}
Let $C_j=\{F : F ~ \mbox{is a face of} ~ Q ~\mbox{and}~ F \cap Q_j ~\mbox{is empty}\}$
and $$U_j = Q - \cup_{F \in C_j} F$$ for $j=1, \ldots, k$. Since $Q$ is a simple polytope,
$U_j$ is diffeomorphic as manifold with corners to $Q_j \times [0, 1)$ and $Q = \cup_{j=1}^k U_{j}$.
Let $$ f_j \colon U_j \to Q_j \times [0,1)$$ be a diffeomorphism. Note that the facets of $Q_j$ are
$\{Q_j \cap F_{j_1}, \ldots, Q_{j} \cap F_{j_l}\}$ for some facets $F_{j_1}, \ldots, F_{j_l} \in \{F_1, \ldots, F_m\}$.
The restriction of the isotropy function $\lambda$ on the facets of $Q_j$ is given by
$$\xi^j(Q_j \cap F_{j_i})= \lambda_{j_i} \quad \mbox{for} \quad i = 1, \ldots, l.$$
By definition of $\lambda$, we can see that $\xi^j$ is a di-characteristic function on $Q_j$.
So by Subsection 2.1 of \cite{PS} the space $X(Q_j, \xi^j) =(\TT^{n} \times Q_j)/\sim_b$ is a $2n$-dimensional
quasitoric orbifold for $j=1, \ldots, k$.
From the equivalence relation $\sim_b$ in (\ref{equilam}), we have the following commutative diagram where
lower horizontal maps are homeomorphisms.
$$
\begin{CD}
\TT^n \times U_j @>Id \times f_i>> \TT^n \times Q_j \times [0, 1) @.\\
@VVV  @VVV @.\\
(\TT^{n} \times U_j)/\sim_b @>h_i>> ((\TT^{n} \times Q_j)/ \sim_b)\times [0,1) @>\cong>> X(Q_j, \xi^j) \times [0, 1).
\end{CD}
$$
So $$W(Q, \lambda)~ = ~ \bigcup_{j=1}^k (T^{n} \times U_j)/\sim_b ~\cong ~\bigcup_{j=1}^k (X(Q_j, \xi^j) \times [0, 1)).$$
Hence $W(Q, \lambda)$ is an orbifold with boundary where the boundary is the disjoint union of quasitoric orbifolds
$\{X(Q_j, \xi^j) \colon j = 1, \ldots, k\}$. Clearly orientations of $\mathbb{T}^{n}$ and $Q$ induce an
orientation of $W(Q, \lambda)$.
\end{proof}

Suppose $\lambda$ satisfies the following condition $\colon$ the set of vectors $ \{\lambda(F_{i_1}), \ldots,$ $ \lambda(F_{i_{l}})\}$
is a part of a basis of $\ZZ^{n}$ over $\ZZ$ whenever the intersection of the facets $ \{F_{i_1}, \ldots,$ $ F_{i_l}\}$
is nonempty. Then all the quasitoric orbifolds in the proof of Lemma \ref{orbbd} are quasitoric manifolds. So, in this
case we have the following corollary.
\begin{corollary}\label{cobm}
With the assumption in the above paragraph, the space $ W(Q, \lambda) $ is a $(2n+1)$-dimensional oriented manifold
with boundary. The boundary is a disjoint union of quasitoric manifolds.
\end{corollary}

\begin{example}\label{egchar2}
Some isotropy functions on the polytopes $Q$ and $Q^{\prime}$ with exceptional facets are given in
the Figure \ref{egc2}. In the left picture of Figure \ref{egc2}, $Q_1, Q_2, Q_3, Q_4$ are exceptional
facets which are triangles. The restriction of the isotropy function on $Q_{i}$ gives that the space
$(\TT^2 \times Q_i)/\sim_b$ is a complex orbifold projective space for $i\in \{1,2,3,4\}$.
So $W(Q, \lambda)$ is an oriented orbifold with boundary where the boundary is the disjoint union
of distinct $4$-dimensional complex orbifold projective spaces.

In the right picture of Figure \ref{egc2}, $Q_1,Q_{2}, Q_{3}, Q_{4}$ and $Q_5$ are exceptional facets where $Q_1, \ldots, Q_4$
are triangles and $Q_{5}$ is a rectangle. The restriction of the isotropy function on $Q_{i}$ gives that the space
$M_i = (\TT^2 \times Q_i)/\sim_b$ is a complex orbifold projective space for $i\in \{1,2,3,4\}$ and
$X(Q_5, \lambda^5) = (\TT^2 \times Q_5)/\sim_b$ is a quasitoric orbifold. Hence the space
$\sqcup_{i=1}^4 M_i \sqcup X(Q_5, \lambda^5)$ is the boundary of the oriented orbifold $W(Q^{\prime}, \lambda)$.
\begin{figure}[ht]
        \centerline{
           \scalebox{0.70}{
            \input{egc2.pstex_t}
            }
          }
       \caption {Isotropy functions of some polytopes with exceptional facets.}
        \label{egc2}
      \end{figure}
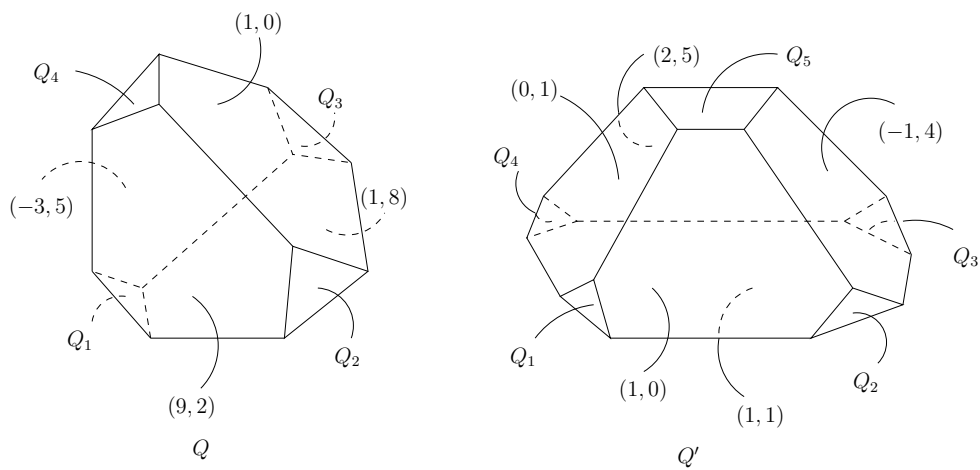
      \qed
\end{example}

\section{Stable complex structure and complex cobordism}\label{ccob}
Stable complex structure of quasitoric manifolds and quasitoric orbifolds are studied in \cite{DJ} and \cite{PS}
respectively.
In this section first we show the existence of stable complex structure on orbifolds with quasitoric boundary $W(Q, \lambda)$.
Then we compute the complex orbifold cobordism class of a quasitoric orbifold explicitly.
At the end of this section we give some computation in $\Omega^U$.
Similarly as the manifolds case, we may define complex cobordism of orbifolds.
\begin{defn}\label{bdm}
Let $Y$ be a topological space and $X_1$, $X_2$ be two $n$-dimensional stable complex orbifolds.
Let $ h_i \colon X_i \to Y$ be a continuous map for $i=1, 2$. Then $h_1$ and $h_2$ are bordant if there exists a
stable complex orbifold $Z$ of dimension $n+1$ with $\partial{Z} = X_1 \sqcup X_2$
and a continuous map $H \colon Z \to Y$ such that $H|_{\partial{Z}}= h_1 \sqcup h_2$. 
\end{defn}
So the Definition \ref{bdm} induces an equivalence relation on the collection
$$\{(X, h) \colon X~ \mbox{is a stable complex orbifold and}~  h \colon X \to Y ~ \mbox{is a continuous map}\}.$$
We denote the equivalence class of $(X, h)$ by $[X, h]$ or $[X]$ if the map $h$ and the stable complex structure
on $X$ are clear.
Let $OB_n^{U}(Y)= \{[X, h] \colon \dim{X} =n\}$. The disjoint union induces an abelian group structure
on $OB_n^{U}(Y)$. The group $OB_{n}^U(Y)$ is called the $n$-th complex orbifold bordism group of $Y$.
Let $OB_{\ast}^{U}(Y) = \cup_{n} OB_n^{U}(Y)$. Then the Cartesian product endow the structure
of a graded ring on $OB_{\ast}^{U}(Y)$, called the complex orbifold bordism ring of $Y$.
\begin{defn}
The complex orbifold bordism groups and ring of a point are called complex orbifold cobordism groups and ring respectively.
\end{defn}
At this moment we do not know about the generators of the group $OB_n^U(Y)$ as well as many other questions
which may arise from the theory of complex cobordism for manifolds. However the complex cobordism ring
$\Omega^U$ is a subring of $OB_{\ast}^U(pt)$. We adhere the notations of Section \ref{def}.

\begin{lemma}\label{qto}
The orbifold with boundary $W(Q, \lambda)$ is an orbit space of a circle action on a quasitoric orbifold.
\end{lemma}
\begin{proof}
Let $\{Q \backslash Q_1, \ldots, Q_k\} $ be a simple $(n+1)$-polytope with exceptional facets
and $$ \mathcal{F}({Q}) = \{ F_{i} \colon i= 1, \ldots, m\} \cup \{Q_j \colon j = 1, \ldots, k\}.$$
Let $\lambda$ be an isotropy function on $\{Q \backslash Q_1, \ldots, Q_{k}\}$.
We define a function $ \eta \colon \mathcal{F}(Q) \to \ZZ^{n+1}$ as follows,
\begin{equation}\label{charp}
\eta( F ) = \left\{ \begin{array}{ll} ( 0, \ldots, 0, 1 ) \in \ZZ^{n+1} & \mbox{if} ~ F = Q_j~ \mbox{and} ~ j~ \in \{1, \ldots, k\} \\
{(\lambda_i, 0)} \in \ZZ^{n} \times\{0\} \subset \ZZ^{n+1} & \mbox{if}~ F = F_i ~ \mbox{and} ~ i \in \{ 1, \ldots, m\}.
\end{array} \right.
\end{equation}
So the function $ \eta $ is a di-characteristic function on $Q$. Let $X(Q, \eta)$ be the 
quasitoric orbifold constructed from the characteristic model $(Q, \eta)$.
There is a natural $ \TT^{n+1} $-action on $X(Q, \eta)$, see  Subsection 2.1 of \cite{PS}. Let
\begin{equation}\label{eq2}
\pi_Q \colon X(Q, \eta) \to Q
\end{equation}
be the orbit map of this action and $ \TT_Q $
be the circle subgroup of $\TT^{n+1}$ determined by the submodule $\{0\} \times \ldots \times \{0\} \times \ZZ $ of
$\ZZ^{n+1}$. From the definition of $\eta$ it is clear that $W(Q, \lambda ) $ is the orbit space of the circle
$ \TT_Q $ action on $ X( Q, \eta )$.
\end{proof}
\begin{remark}
The quotient map $ \phi_Q \colon X( Q, \eta ) \to W( Q, \lambda) $ is not a fiber bundle map.
\end{remark}

\begin{theorem}\label{comcob}
Let $ \{Q \backslash Q_1, \ldots, Q_k\} $ be a simple $(n+1)$-polytope with exceptional facets in $\RR^{n+1}$ and 
 $\lambda$ be an isotropy function on $\{Q \backslash Q_1, \ldots, Q_k\}$. Then there is a stable
complex structure on $W(Q, \lambda)$. Moreover $[X(Q_1, \xi^1)] + \cdots + [X(Q_k, \xi^k)] = 0$ in $OB_{2n}^U(pt)$ where $\xi^i$
is the restriction of the isotropy function $\lambda$ on the facets of $Q_j$ for $j=1, \ldots, k$.
\end{theorem}
\begin{proof}
We construct a di-characteristic model $(Q, \eta)$ from the pair $(Q, \lambda)$, see Equation (\ref{charp}).
Let $X(Q, \eta)$ be the quasitoric orbifold constructed from the characteristic model $(Q, \eta)$, see 
Subsection 2.1 of \cite{PS}.
Let $X_1, \ldots, X_m$ be the characteristic suborbifolds of $X(Q, \eta)$ and the 
omniorientation of $X(Q, \eta)$ be the characteristic omniorientation. By Section 6 of \cite{PS},
this omniorientation determines a stably complex structure on $X(Q, \eta)$ by means of the following
isomorphism of orbifold real $2m$-bundles
\begin{equation}\label{scs}
T(X(Q, \eta)) \oplus \RR^{2(m-n-1)} \cong \rho_1 \oplus \cdots \oplus \rho_m
\end{equation}
where the orbifold complex line bundles $\rho_i$'s can be interpreted in the following way.
The orientation of the orbifold normal bundle $\mu_i$ over the characteristic suborbifold $X_i$
defines a rational Thom class in the cohomology group $H^2({\bf T}(\mu_i), \QQ)$, represented by
a complex line bundle over the Thom complex ${\bf T} (\mu_i)$. We pull this back along the
Pontryagin-Thom collapse $X(Q, \eta) \to {\bf T}(\mu_i)$, and denote the resulting orbibundle by $\rho_i$. 

Cut off a neighborhood of each facets $ Q_j$ of $Q$
by an affine hyperplane $ H_j$  in $\RR^{n+1}$ for $j= 1, \ldots, k$ such that $H_i \cap H_j \cap Q$
is empty for $i \neq j$, $1 \leq i, j \leq k$. Then the remaining subset of $Q$, denoted by $Q_P$, is an $(n+1)$-dimensional
simple polytope which is naturally diffeomorphic as manifold with corners to $Q$. Suppose $$ Q_{H_j} = Q \cap H_j =
H_j \cap Q_P$$ for $ j=1, \ldots k $. Then $ Q_{H_j}$ is a facet of $Q_P$
for each $j \in \{1, \ldots, k\}$.
Note that $Q_{H_j}$ is diffeomorphic as manifold with corners to $Q_j$ for $j=1, \ldots, k$.
Clearly, $W(Q_P, \lambda)= (\TT^{n} \times Q_P)/\sim_b$ is equivariantly diffeomorphic to $W(Q, \lambda)$.
Let $\bar{W}$ be the pullback of the following diagram $\colon$
\begin{equation}
\begin{CD}
\bar{W} @>>> X(Q, \eta)\\
@V\pi_v VV  @V\pi_Q VV \\
Q_P @>\iota >> Q.
\end{CD}
\end{equation}
 Then
$\bar{W} = W(Q_P, \lambda) \times \TT_Q$ where $\TT_Q$ is the circle subgroup of $\TT^{n+1}$ determined by
the vector $(0, \ldots, 0, 1)$ in $\ZZ^{n+1}$ and $\pi_Q$ is given in Equation (\ref{eq2}). We have the 
following commutative diagrams of complex orbibundles.
\begin{equation}
\begin{CD}
E_3 @>>> E_2 @>>> E_1 @>>> T(X(Q, \eta)) \oplus \RR^{2(m-n-1)}\\
@VVV @VVV @VVV  @VVV \\
X(Q_i, \xi^i) @>>>W(Q_P, \lambda) @>>> \bar{W} @>>> X(Q, \eta)
\end{CD}
\end{equation}
where $E_1, E_2$ and $E_3$ are the pullback bundles. Since $\bar{W} = W(Q_P, \lambda) \times T_Q$, we have the following isomorphism
of bundles $\colon$ $$E_2 \cong T(W(Q_P, \lambda)) \oplus \RR^{2(m-n)-1}.$$ Hence for a choice of sign of isotropy
vectors of $\{Q \backslash Q_1, \ldots, Q_k\}$ the orbifold
with boundary $W(Q, \lambda)$ has a stable complex structure.

Observe that the bundle $E_3$ is isomorphic to
\begin{equation}
T(X(Q_i, \xi^i)) \oplus \RR^{2(m-n)} \cong \rho_{i_1} \oplus \cdots \oplus \rho_{i_l} \oplus \RR^{2 (m-l)}
\end{equation}
where $\rho_{i_j}$ is the complex line bundle corresponding to the $i_j$-th characteristic suborbifold of $X(Q_i, \xi^i)$.
Hence $[X(Q_1, \xi^1)] + \cdots + [X(Q_k, \xi^k)] = 0$ in the orbifold complex cobordism group $OB_{2n}^U(pt)$.
\end{proof}

\begin{theorem}\label{qbd}
Let $X$ be a $2n$-dimensional omnioriented quasitoric orbifold over the simple polytope $P$. Then $[X] = [M_1]
+ \cdots + [M_k]$ in $OB_{2n}^U(pt)$, where $M_1, \ldots, M_k$ are complex orbifold projective spaces.
\end{theorem}
\begin{proof}
Let $\mathcal{F}(P) = \{F_1, \ldots, F_m\}$ be the facets and $\{v_1, \ldots, v_k\}$ be the vertices
of $P \subset \RR^n$.  Let $Q = P \times [0, 1]$. So $Q$ is an $(n+1)$-dimensional simple polytope in
$\RR^{n+1}$. Cut off a neighborhood of each vertex $v_j \times \{0\}$ of $Q$ by an affine hyperplane $H_j$
for $j=1, \ldots, k$ in $\RR^{n+1}$ such that
$$H_i \cap H_j \cap Q ~ \mbox{is empty for} ~ i \neq j, ~ \mbox{and} ~ H_j \cap P ~\mbox{is empty for}
~i, j \in \{1, \ldots, k\}.$$
Then the remaining subset of $Q$, denoted by $Q_P$, is an $(n+1)$-dimensional simple polytope. Observe
that $\bigtriangleup_i^n = Q \cap H_i = H_i \cap Q_P$ is a facet of $Q_P$. Also $\bigtriangleup_i^n$ is
an $n$-dimensional simplex
in $\RR^{n+1}$ for $i=1, \ldots, k$. So $\{Q_P \backslash P \times \{1\}, \bigtriangleup^n_1, \ldots, \bigtriangleup^n_k \}$
is an $(n+1)$-dimensional simple polytope with exceptional facets.
Let
$$\bar{F}_0 = Q_P \cap (P \times \{0\}), ~\mbox{and} ~ \bar{F}_i = Q_P \cap (F_i \times [0, 1]) ~\mbox{for} ~ i=1, \ldots, m.$$
So $\mathcal{F}(Q_P)=\{\bar{F}_i \colon i=0, \ldots, m\} \cup \{P \times \{1\}, \bigtriangleup^n_1, \ldots, \bigtriangleup^n_k\}$.

Let $\xi \colon \mathcal{F}(P) \to \ZZ^n$ be the di-characteristic function associated to the omniorientation of $X$.
Let $E(P)$ be the set of all edges of $P$ and $e \in E(P)$. Then $e = F_{i_1} \cap \cdots \cap F_{i_{n-1}}$ for
a unique collection of facets $F_{i_1}, \ldots, F_{i_{n-1}}$ of $P$.
Let $\ZZ(e)$ be the submodule of $\ZZ^n$ generated by $\{\xi(F_{i_1}), \ldots, \xi(F_{i_{n-1}})\}$. So $\ZZ(e)$ is a
free $\ZZ$-module of rank $n-1$ for any $e \in E(P)$. From Remark \ref{re1}, there exists $\lambda_0 \in
\ZZ^n - \cup_{e \in E(P)} \ZZ(e)$. So $\lambda_0$ is nonzero and
 $\{\lambda_0, \xi(F_{i_1}), \ldots, \xi(F_{i_{n-1}})\}$ is a linearly independent set in $\ZZ^n$ for the edge
$e= F_{i_1} \cap \cdots \cap F_{i_{n-1}}$. 
Define $\lambda \colon \{\bar{F}_i \colon i=0, \ldots, m\} \to \ZZ^n$ by
\begin{equation}
\lambda(F) = \left\{ \begin{array}{ll} \xi(F_i) & \mbox{if}~  F= \bar{F}_i \\
\lambda_0 & \mbox{if}~ F = \bar{F}_0.
 \end{array} \right.
\end{equation}
So the function $\lambda$ is an isotropy function on $\{Q_P \backslash P \times \{1\}, 
\bigtriangleup^n_1, \ldots, \bigtriangleup^n_k\}$. 
Let $$\xi^j \colon \mathcal{F}(\bigtriangleup^n_j) \to \ZZ^n$$ be the restriction of $\lambda$ on the facets
of $\bigtriangleup^n_j$ for $j=1, \ldots, k$. Then $(\bigtriangleup_j^n, \xi^j)$ is a characteristic model
for $j=1, \ldots, k$. Let $M_j$ be the complex orbifold projective space constructed from the characteristic model
$(\bigtriangleup^n_j, \xi^j)$ for $j=1, \ldots, k$. So by Lemma \ref{orbbd}, the boundary of
$(2n+1)$-dimensional oriented orbifold $W(Q_P, \lambda)$ is $X \sqcup M_1 \sqcup \cdots \sqcup M_k$.
 Hence by Theorem \ref{comcob} we get the orbifold complex cobordism relation $[X]= [M_1] + \cdots + [M_k]$
 in $OB_{2n}^U(pt)$.
\end{proof}

\begin{example} Let $X$ be a 2-dimensional quasitoric orbifold over $P$. Then $P$ is a closed interval say $[0, 1]$
and the corresponding di-characteristic function on $P$ is given by $\xi(\{0\})= p_1$, $\xi(\{1\})=p_2$ (possibly up to
sign) for some nonzero integers $p_1$ and $p_2$. So the Figure \ref{egtd} gives an isotropy function $\lambda$ on
$\{Q_P \backslash P \times 1, \bigtriangleup_1^1, \bigtriangleup_2^1\}$. Then $W(Q_P, \lambda)$ is a 3-dimensional
stably complex orbifold with boundary $X \sqcup M_1 \sqcup M_2$ where $M_i$ is the quasitoric orbifold corresponding to the facet
$\bigtriangleup_i^1$ for $i=1,2$. By Example \ref{td} and Proposition \ref{probi}, $M_i$ is equivariantly homeomorphic
to the teardrop $\mathbb{WP}(1, p_i)$ for $i=1, 2$. Therefore $X$ is complex orbifold cobordant to two copies of teardrop.
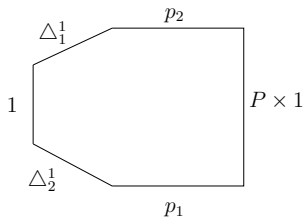
\begin{figure}[ht]
        \centerline{
           \scalebox{0.70}{
            \input{egtd.pstex_t}
            }
          }
       \caption {An isotropy function on a $2$-polytope with exceptional facets.}
        \label{egtd}
      \end{figure}
      \qed
\end{example}

\begin{obser}\label{obs}
Let $$A = \{[M] \colon M ~ \mbox{is a} ~2n\mbox{-dimensional complex orbifold projective space}\}.$$ We show that $A$ is not a
linearly independent set in $OB_{2n}^U(pt)$. Let $Q$ be an $(n+1)$-dimensional simple
polytope in $\RR^{n+1}$ with vertices $\{v_1, \ldots, v_k\}$ and facets $\{F_1, \ldots, F_m\}$. We delete
a neighborhood of each vertex $v_i$ by cutting with an hyperplane $H_i$ in $\RR^{n+1}$ for $i= 1, \ldots, k$ 
such that $H_i \cap H_j \cap Q$ is empty for $i \neq j$, $1 \leq i, j \leq k$. Let $Q_V$ be the remaining subset of $Q$
and $$\bigtriangleup_i^n = Q \cap H_i = Q_V \cap H_i \quad \mbox{for} \quad i=1, \ldots, k.$$
Since $Q$ is an $(n+1)$-dimensional simple polytope, $\bigtriangleup_i^n$ is an $n$-simplex
for $i=1, \ldots, k$.  Let $\bar{F}_i = F_i \cap Q_V$ for $i=1, \ldots, m$. 
So $\{Q_V \backslash \bigtriangleup_1^n, \ldots, \bigtriangleup_k^n\}$ is an 
$(n+1)$-dimensional simple polytope with exceptional facets and $$\mathcal{F}(Q_V)=\{\bar{F}_1, \ldots, \bar{F}_m\}
\cup \{\bigtriangleup_1^n, \ldots, \bigtriangleup_k^n\}.$$
Since $\ZZ^n$ is not a union of finitely many proper submodules over $\ZZ$, we can define an
isotropy function $$\lambda \colon \{\bar{F}_1, \ldots, \bar{F}_m\} \to \ZZ^n$$ on $\{Q_V \backslash 
\bigtriangleup_1^n, \ldots, \bigtriangleup_k^n\}$.
Since $Q$ is an $(n+1)$-dimension simple polytope, each vertex $v_i$ of $Q$ is the intersection of unique collection
of facets $\{F_{i_1}, \ldots, F_{i_{n+1}}\}$ for $i=1, \ldots, k$. So 
$\mathcal{F}(\bigtriangleup_i^n) = \{F_{i_j} \cap \bigtriangleup_i^n \colon j = 1, \ldots, n+1\}$. We define a map
$\xi^i \colon \mathcal{F}(\bigtriangleup_i^n) \to \ZZ^n$ by
\begin{equation}
\xi^i(F_{i_j} \cap \bigtriangleup_i^n) = \lambda(\bar{F}_{i_j}) \quad \mbox{for} \quad j=1, \ldots, n+1.
\end{equation}
Then $\xi^i$ is a di-characteristic function on $\bigtriangleup_i^n$. Let $X(\bigtriangleup_i^n, \xi^i)$ be the 
complex orbifold projective space for the characteristic model $(\bigtriangleup_i^n, \xi^i)$ for $i=1, \ldots, k$. 
So by Lemma \ref{orbbd} the space $W(Q_V, \lambda)$ is an oriented orbifold with boundary where the boundary
$\partial{W(Q_V, \lambda)}$ is a disjoint union of $\{X(\bigtriangleup_i^n, \xi^i) \colon i=1, \ldots, k\}$.
Thus by Theorem \ref{comcob}, we have
$[X(\bigtriangleup_i^n, \xi^1)] + \cdots + [X(\bigtriangleup_i^n, \xi^k)] =0$ in $OB_{2n}^U(pt)$.
Hence the set of vectors in $A$ is not linearly independent.
This also follows from Theorem \ref{qbd} if $X$ is a complex orbifold projective space.
In that case $k$ is $n+1$. But if $Q$ is any $(n+1)$-dimensional simple polytope, $k$ can be as large as possible.
So we get more relations among complex orbifold projective spaces.
Now we may ask the following question.
\end{obser}
\begin{ques}
Are there any other types of relation among complex orbifold projective spaces, and if so how do they arise?
\end{ques}

Now we give some computations in the complex cobordism ring $\Omega^U$. Milnor and Novikov independently showed that the ring
$\Omega^U$ is isomorphic to the polynomial ring $\ZZ[a_1, a_2, \ldots]$ where $\deg a_i=2i$, see \cite{Nov}. 
The complex projective spaces $\CP^n $ for $ n \geq 0$ and Milnor hypersurfaces can be chosen as the standard set of multiplicative
generators for $\Omega^U$, see example 5.39 in \cite{BP}. 
In \cite{BR1}, Buchstaber and Ray introduced a new set of multiplicative generators for $\Omega^U$ which
are quasitoric manifolds. In \cite{BR2}, they proved that any complex cobordism class contains a quasitoric
manifold in dimension $>$ 2 by showing that a disjoint union of products of this new generators is complex cobordant
to a quasitoric manifold. See example 5.28 in \cite{BP} for the case of dimension 2.

\begin{example}\label{cncpn}
 Let $\bigtriangleup^n$ be the $n$-simplex with $\mathcal{F}(\bigtriangleup^n)=\{F_0, \ldots, F_n\}$
 and $$\xi \colon \mathcal{F}(\bigtriangleup^n) \to \ZZ^n$$ be the characteristic function such that
 the set $\{\xi_0, \ldots, \widehat{\xi_i}, \ldots, \xi_n \}$ is a basis of $\ZZ^n$ where
 $\widehat{}$ represents the omission of the corresponding entry. Let $\xi_i = \xi(F_i)$  for $i=0, \ldots, n$.
 Suppose $\xi_i=(a_{i1}, \ldots,  a_{in})$ for $ i=0, \ldots, n$. Let $X(\bigtriangleup^n, \xi)$ be the quasitoric manifold
 constructed from the characteristic model $(\bigtriangleup^n, \xi)$. Then $X(\bigtriangleup^n, \xi)$ is
 $\delta$-equivariantly homeomorphic to $\CP^n$ for some automorphism $\delta$ of $\TT^n$, see Proposition 5.63 in \cite{BP}.
 The map $\delta$  induces an automorphism $\delta_{\ast}$ of $\ZZ^n$. We may assume that $\delta_{\ast}((\xi_i)^t) = e_i^t$
 where $e_i$ is the $i$-th standard vector of $\ZZ^n$ and $e^t_i$ is the transpose of $e_i$ for $i=1, \ldots, n$.
 By condition on the characteristic function $\xi$
 one can show that the set $\{\delta_{\ast}((\xi_0)^t), \ldots, \widehat{\delta_{\ast}((\xi_i})^t), \ldots,
 \delta_{\ast}((\xi_n)^t) \}$ is a basis of $\ZZ^n$ for $i=0, \ldots, n$. Hence $\delta_{\ast}((\xi_0)^t)=(a_1, \ldots, a_n)^t$
 where $a_i= \pm 1$.
 
 By Theorem 5.18 in \cite{BP} the cohomology ring of
 $X(\bigtriangleup^n, \xi)$ with integer coefficients is $\ZZ[x_0, \ldots, x_n]/{I+J}$ where $I=\< {x_0 ... x_n}\>$
 and $$J=\<\{a_{1i}x_1+ \cdots + a_{ni}x_n + a_{0i}x_0 : i=1, \ldots, n \}\>$$ and $x_i$ is the Poincare 
 dual of the characteristic submanifold corresponding to the facet $F_i$ for $i=1, \ldots, n$.
 So  in the cohomology ring of $X(\bigtriangleup^n, \xi)$, we get a system of homogeneous equation
 $A {\bf x}= {\bf b}x_0$ where $i$-th column of $A$ is $(\xi_i)^t$, ${\bf x}=(x_1, \ldots, x_n)^t$ and
 ${\bf b}=-(\xi_0)^t$. Then $${\bf x}= \delta^{\ast} A {\bf x} = \delta^{\ast} {\bf b} x_0 = - (a_1, \ldots, a_n)^t x_0.$$
  
 Suppose $a_{j}=1$ for $j \in \{i_1, \ldots, i_l\} \subseteq \{1, \ldots, n\}$ and $a_j=-1$ for $j \notin \{i_1, \ldots, i_l\}$.
  So in the cohomology ring of $X(\bigtriangleup^n, \xi)$,
  $x_0^{n+1}=0$, $x_{i_j}=-x_0$ for $i_j \in \{i_1, \ldots, i_l\}$ and $x_j=x_0$ for $j \notin \{i_1, \ldots, i_l\}$.
  Assume the omniorientation of $X(\bigtriangleup^n, \xi)$ 
 is the characteristic omniorientation. Then by Theorem 5.34 in \cite{BP}, the total Chern class of
 the corresponding stable complex bundle on $X(\bigtriangleup^n, \xi)$ is given by
 $$\mathcal{C}_{\xi} = (1+x_0) \cdots (1+x_n) = (1-x_0)^l (1+x_0)^{n+1-l}.$$
 Using the binomial theorem and Cauchy product formula one can compute the coefficient
 of $x_0^i$ in this expression. Let $c_{\xi,i}$ be the coefficient of $x_0^i$ and $K=(k_1, \ldots, k_s)$ be 
 a partition of $n$. Then $K$-th Chern number of $X(\bigtriangleup^n, \xi)$ is
 $ C_{\xi, K} = c_{\xi, k_1} \cdots c_{\xi, k_s}$.  \qed
\end{example}

Now we discuss some computations in the complex cobordism ring $\Omega^U$.
Let $M$ be a $2n$-dimensional omnioriented quasitoric manifold over a simple polytope $P$. Let
$\xi \colon \mathcal{F}(P) \to \ZZ^n$ be the corresponding characteristic function on $P$.
We introduce some combinatorial data in the following.
Let $\{Q \backslash P, \bigtriangleup^n_1, \ldots, \bigtriangleup^n_k\}$ be
an $(n+1)$-dimensional simple polytope with exceptional facets,
where $\bigtriangleup_1^n, \ldots, \bigtriangleup_k^n$ are $n$-dimensional simplices.
Let $$\mathcal{F}(Q) = \{F_1, \ldots, F_m\} \cup \{P, \bigtriangleup^n_1, \ldots, \bigtriangleup^n_k\}$$
and $$\lambda \colon \{F_1, \ldots, F_m\} \to \ZZ^n$$ be an isotropy function on
$\{Q \backslash P, \bigtriangleup^n_1, \ldots, \bigtriangleup^n_k\}$
such that the following holds:
\begin{enumerate}
 \item $\lambda(F_i)=\xi(F_i \cap P)$ if $F_i \cap P$ is nonempty.
\item If $e$ is an edge of $Q$ not contained in $\cup_{i=1}^k \bigtriangleup^n_i \cup P$ then
$\{\lambda(F_{i_1}), \ldots, \lambda(F_{i_n})\}$ is a basis of $\ZZ^n$ where $e = F_{i_1} \cap \cdots \cap F_{i_n}$.
\end{enumerate}
Note that $F_i \cap P$ is nonempty if and only if $F_i \cap P$ is a facet of $P$. So the restriction of
$\lambda$ on the facets of $P$ is the map $\xi$. Thus we may assume $\lambda$ is an extension of $\xi$.

Let $\xi^i \colon \mathcal{F}(\bigtriangleup^n_i) \to \ZZ^n$ be the map defined by 
$$\xi^i(F) = \lambda(F_j) ~ \mbox{if} ~ F= F_j \cap \bigtriangleup^n_i.$$
So $\xi^i$ is a characteristic function on $\bigtriangleup_i^n$ for $i=1, \ldots, k$. Let $X(\bigtriangleup^n_i, \xi^i)$
be the quasitoric manifold constructed from the characteristic model $(\bigtriangleup_i^n, \xi^i)$ for $i=1, \ldots, k$.
Recall that $X(\bigtriangleup_i^n, \xi^i)$ is $\delta$-equivariantly diffeomorphic to $\CP^n$ for some
$\delta \in \Aut(\TT^n)$. Then by Corollary \ref{cobm}, the space $W(Q, \lambda)$ is a manifold with 
quasitoric boundary where the boundary is $M \sqcup X(\bigtriangleup_1^n, \xi^1) \sqcup \ldots \sqcup X(\bigtriangleup_k^n, \xi^k)$.
Therefore by Theorem \ref{qbd} we get $$[M] = k[\CP^n]$$ in $\Omega^U$, where the stable complex structure on each $\CP^n$ is
determined by the corresponding characteristic function.

Let $K$ be a partition of $n$ and $C_{i, K}$ be the $K$-th Chern number of $X(\bigtriangleup_i^n, \xi^i)$
for $i=1, \ldots, k$. Since two stably complex manifolds are cobordant if and only if their Chern
numbers are identical (by Milnor \cite{Mil} and Novikov \cite{Nov2}), we get the following 
formula for the $K$-th Chern number $C_K$ of $M$, $$C_K (M) = C_{1,K} + \cdots + C_{k, K}.$$ 
Recall that every quasitoric manifold has a stable complex structure which depends on the omniorientation
on it. Not all quasitoric manifold admit an almost complex structure. For example,
$\CP^2 \# \CP^2$ is a quasitoric manifold, but not an almost complex manifold.

\begin{theorem}\label{hp}
In the above discussion, if $M$ is an almost complex quasitoric manifold then the complex cobordism class of
$\sqcup_{i=1}^k X(\bigtriangleup^n_i, \xi^i)$ contains the almost complex quasitoric manifold.
\end{theorem}
\begin{remark}
In the Theorem \ref{hp}, if $M$ is a smooth projective space then we get a representation of $[M]$
in term of some generators of complex cobordism ring $\Omega^U$. Therefore the above combinatorial
process  gives some sufficient conditions for the Hirzebruch problem which is mentioned in the introduction.
\end{remark}

\begin{example}
The quasitoric manifold corresponding to the characteristic function on $P$ in Figure \ref{egc3} is the Hirzebruch
surface $M^4_2$, see Example 1.19 in \cite{DJ}. The function on $\{Q \backslash Q_1, \ldots, Q_5\}$ is an isotropy
function which extends the characteristic function on $P \cong Q_5$. Observe that the restriction of the isotropy
function on the facets of $\bigtriangleup^2_i$ is a characteristic function $\xi^i \colon \mathcal{F}(\bigtriangleup_i^2)
\to \ZZ^2$ for $i=1, \ldots, 4$. For each $i \in \{1, \ldots, 4\}$, the corresponding quasitoric manifold
$X(\bigtriangleup_i^2, \xi^i)$ is $\delta$-equivariantly homeomorphic to $\CP^2$. Then $[M_2^2] = 4 [\CP^2]$ where 
the stable complex structure on $\CP^2$ is determined by the corresponding characteristic function. Hence
the complex cobordism class of $4 [\CP^2]$ contains a connected algebraic variety.
\begin{figure}[ht]
        \centerline{
           \scalebox{0.70}{
            \input{egc3.pstex_t}
            }
          }
       \caption {A characteristic and isotropy function on a $2$- and $3$-polytope.}
        \label{egc3}
      \end{figure}
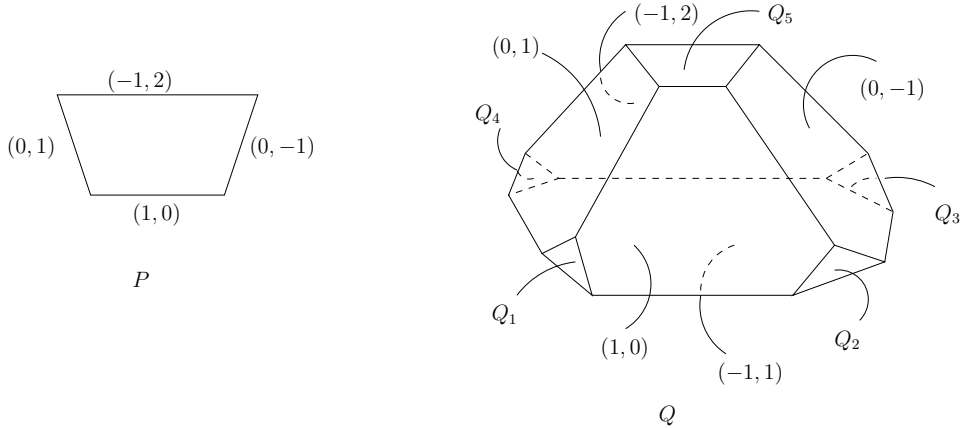
      \qed
\end{example}

{\bf Acknowledgement.} The author would like to thank Professor Mainak Poddar, Marzieh Bayeh and
Professor Goutam Mukherjee for helpful suggestions and stimulating discussions.
He also thanks anonymous referee for many helpful suggestions.
He thanks University of Regina and Pacific Institute for the Mathematical Sciences
for support. 

\bibliographystyle{alpha}
\bibliography{bibliography.bib}

\vspace{1cm}

\vfill

\end{document}

%% file: egc2.pstex_t
\begin{picture}(0,0)%
\includegraphics{egc2.pstex}%
\end{picture}%
\setlength{\unitlength}{3947sp}%
\begingroup\makeatletter\ifx\SetFigFont\undefined%
\gdef\SetFigFont#1#2#3#4#5{%
  \reset@font\fontsize{#1}{#2pt}%
  \fontfamily{#3}\fontseries{#4}\fontshape{#5}%
  \selectfont}%
\fi\endgroup%
\begin{picture}(8505,4116)(586,-5155)
\put(7126,-4711){\makebox(0,0)[lb]{\smash{{\SetFigFont{12}{14.4}{\rmdefault}{\mddefault}{\updefault}{\color[rgb]{0,0,0}$(1,1)$}%
}}}}
\put(6076,-4486){\makebox(0,0)[lb]{\smash{{\SetFigFont{12}{14.4}{\rmdefault}{\mddefault}{\updefault}{\color[rgb]{0,0,0}$(1,0)$}%
}}}}
\put(5101,-1786){\makebox(0,0)[lb]{\smash{{\SetFigFont{12}{14.4}{\rmdefault}{\mddefault}{\updefault}{\color[rgb]{0,0,0}$(0,1)$}%
}}}}
\put(8401,-2161){\makebox(0,0)[lb]{\smash{{\SetFigFont{12}{14.4}{\rmdefault}{\mddefault}{\updefault}{\color[rgb]{0,0,0}$(-1,4)$}%
}}}}
\put(6376,-1486){\makebox(0,0)[lb]{\smash{{\SetFigFont{12}{14.4}{\rmdefault}{\mddefault}{\updefault}{\color[rgb]{0,0,0}$(2,5)$}%
}}}}
\put(5101,-4186){\makebox(0,0)[lb]{\smash{{\SetFigFont{12}{14.4}{\rmdefault}{\mddefault}{\updefault}{\color[rgb]{0,0,0}$Q_1$}%
}}}}
\put(8176,-4411){\makebox(0,0)[lb]{\smash{{\SetFigFont{12}{14.4}{\rmdefault}{\mddefault}{\updefault}{\color[rgb]{0,0,0}$Q_2$}%
}}}}
\put(9076,-3286){\makebox(0,0)[lb]{\smash{{\SetFigFont{12}{14.4}{\rmdefault}{\mddefault}{\updefault}{\color[rgb]{0,0,0}$Q_3$}%
}}}}
\put(7576,-1486){\makebox(0,0)[lb]{\smash{{\SetFigFont{12}{14.4}{\rmdefault}{\mddefault}{\updefault}{\color[rgb]{0,0,0}$Q_5$}%
}}}}
\put(4951,-2386){\makebox(0,0)[lb]{\smash{{\SetFigFont{12}{14.4}{\rmdefault}{\mddefault}{\updefault}{\color[rgb]{0,0,0}$Q_4$}%
}}}}
\put(2026,-4636){\makebox(0,0)[lb]{\smash{{\SetFigFont{12}{14.4}{\rmdefault}{\mddefault}{\updefault}{\color[rgb]{0,0,0}$(9, 2)$}%
}}}}
\put(3751,-2761){\makebox(0,0)[lb]{\smash{{\SetFigFont{12}{14.4}{\rmdefault}{\mddefault}{\updefault}{\color[rgb]{0,0,0}$(1, 8)$}%
}}}}
\put(601,-2836){\makebox(0,0)[lb]{\smash{{\SetFigFont{12}{14.4}{\rmdefault}{\mddefault}{\updefault}{\color[rgb]{0,0,0}$(-3, 5)$}%
}}}}
\put(2626,-1186){\makebox(0,0)[lb]{\smash{{\SetFigFont{12}{14.4}{\rmdefault}{\mddefault}{\updefault}{\color[rgb]{0,0,0}$(1, 0)$}%
}}}}
\put(1126,-4036){\makebox(0,0)[lb]{\smash{{\SetFigFont{12}{14.4}{\rmdefault}{\mddefault}{\updefault}{\color[rgb]{0,0,0}$Q_1$}%
}}}}
\put(3526,-4186){\makebox(0,0)[lb]{\smash{{\SetFigFont{12}{14.4}{\rmdefault}{\mddefault}{\updefault}{\color[rgb]{0,0,0}$Q_2$}%
}}}}
\put(3376,-1861){\makebox(0,0)[lb]{\smash{{\SetFigFont{12}{14.4}{\rmdefault}{\mddefault}{\updefault}{\color[rgb]{0,0,0}$Q_3$}%
}}}}
\put(826,-1636){\makebox(0,0)[lb]{\smash{{\SetFigFont{12}{14.4}{\rmdefault}{\mddefault}{\updefault}{\color[rgb]{0,0,0}$Q_4$}%
}}}}
\put(2251,-5011){\makebox(0,0)[lb]{\smash{{\SetFigFont{12}{14.4}{\rmdefault}{\mddefault}{\updefault}{\color[rgb]{0,0,0}$Q$}%
}}}}
\put(6601,-5086){\makebox(0,0)[lb]{\smash{{\SetFigFont{12}{14.4}{\rmdefault}{\mddefault}{\updefault}{\color[rgb]{0,0,0}$Q^{\prime}$}%
}}}}
\end{picture}%

%% file: egtd.pstex_t
\begin{picture}(0,0)%
\includegraphics{egtd.pstex}%
\end{picture}%
\setlength{\unitlength}{4144sp}%
\begingroup\makeatletter\ifx\SetFigFont\undefined%
\gdef\SetFigFont#1#2#3#4#5{%
  \reset@font\fontsize{#1}{#2pt}%
  \fontfamily{#3}\fontseries{#4}\fontshape{#5}%
  \selectfont}%
\fi\endgroup%
\begin{picture}(2100,1876)(1561,-2600)
\put(1576,-1681){\makebox(0,0)[lb]{\smash{{\SetFigFont{12}{14.4}{\rmdefault}{\mddefault}{\updefault}{\color[rgb]{0,0,0}$1$}%
}}}}
\put(1846,-1051){\makebox(0,0)[lb]{\smash{{\SetFigFont{12}{14.4}{\rmdefault}{\mddefault}{\updefault}{\color[rgb]{0,0,0}$\bigtriangleup_1^1$}%
}}}}
\put(1756,-2311){\makebox(0,0)[lb]{\smash{{\SetFigFont{12}{14.4}{\rmdefault}{\mddefault}{\updefault}{\color[rgb]{0,0,0}$\bigtriangleup_2^1$}%
}}}}
\put(3646,-1636){\makebox(0,0)[lb]{\smash{{\SetFigFont{12}{14.4}{\rmdefault}{\mddefault}{\updefault}{\color[rgb]{0,0,0}$P \times {1}$}%
}}}}
\put(2926,-2536){\makebox(0,0)[lb]{\smash{{\SetFigFont{12}{14.4}{\rmdefault}{\mddefault}{\updefault}{\color[rgb]{0,0,0}$p_1$}%
}}}}
\put(2926,-871){\makebox(0,0)[lb]{\smash{{\SetFigFont{12}{14.4}{\rmdefault}{\mddefault}{\updefault}{\color[rgb]{0,0,0}$p_2$}%
}}}}
\end{picture}%

%% file: egc3.pstex_t
\begin{picture}(0,0)%
\includegraphics{egc3.pstex}%
\end{picture}%
\setlength{\unitlength}{3947sp}%
\begingroup\makeatletter\ifx\SetFigFont\undefined%
\gdef\SetFigFont#1#2#3#4#5{%
  \reset@font\fontsize{#1}{#2pt}%
  \fontfamily{#3}\fontseries{#4}\fontshape{#5}%
  \selectfont}%
\fi\endgroup%
\begin{picture}(8355,3807)(736,-5146)
\put(6076,-4486){\makebox(0,0)[lb]{\smash{{\SetFigFont{12}{14.4}{\rmdefault}{\mddefault}{\updefault}{\color[rgb]{0,0,0}$(1,0)$}%
}}}}
\put(5101,-1786){\makebox(0,0)[lb]{\smash{{\SetFigFont{12}{14.4}{\rmdefault}{\mddefault}{\updefault}{\color[rgb]{0,0,0}$(0,1)$}%
}}}}
\put(5101,-4186){\makebox(0,0)[lb]{\smash{{\SetFigFont{12}{14.4}{\rmdefault}{\mddefault}{\updefault}{\color[rgb]{0,0,0}$Q_1$}%
}}}}
\put(8176,-4411){\makebox(0,0)[lb]{\smash{{\SetFigFont{12}{14.4}{\rmdefault}{\mddefault}{\updefault}{\color[rgb]{0,0,0}$Q_2$}%
}}}}
\put(9076,-3286){\makebox(0,0)[lb]{\smash{{\SetFigFont{12}{14.4}{\rmdefault}{\mddefault}{\updefault}{\color[rgb]{0,0,0}$Q_3$}%
}}}}
\put(7576,-1486){\makebox(0,0)[lb]{\smash{{\SetFigFont{12}{14.4}{\rmdefault}{\mddefault}{\updefault}{\color[rgb]{0,0,0}$Q_5$}%
}}}}
\put(4951,-2386){\makebox(0,0)[lb]{\smash{{\SetFigFont{12}{14.4}{\rmdefault}{\mddefault}{\updefault}{\color[rgb]{0,0,0}$Q_4$}%
}}}}
\put(6601,-5086){\makebox(0,0)[lb]{\smash{{\SetFigFont{12}{14.4}{\rmdefault}{\mddefault}{\updefault}{\color[rgb]{0,0,0}$Q$}%
}}}}
\put(8401,-2161){\makebox(0,0)[lb]{\smash{{\SetFigFont{12}{14.4}{\rmdefault}{\mddefault}{\updefault}{\color[rgb]{0,0,0}$(0, -1)$}%
}}}}
\put(6376,-1486){\makebox(0,0)[lb]{\smash{{\SetFigFont{12}{14.4}{\rmdefault}{\mddefault}{\updefault}{\color[rgb]{0,0,0}$(-1, 2)$}%
}}}}
\put(7126,-4711){\makebox(0,0)[lb]{\smash{{\SetFigFont{12}{14.4}{\rmdefault}{\mddefault}{\updefault}{\color[rgb]{0,0,0}$(-1,1)$}%
}}}}
\put(1876,-3886){\makebox(0,0)[lb]{\smash{{\SetFigFont{12}{14.4}{\rmdefault}{\mddefault}{\updefault}{\color[rgb]{0,0,0}$P$}%
}}}}
\put(1876,-3286){\makebox(0,0)[lb]{\smash{{\SetFigFont{12}{14.4}{\rmdefault}{\mddefault}{\updefault}{\color[rgb]{0,0,0}$(1, 0)$}%
}}}}
\put(2926,-2686){\makebox(0,0)[lb]{\smash{{\SetFigFont{12}{14.4}{\rmdefault}{\mddefault}{\updefault}{\color[rgb]{0,0,0}$(0, -1)$}%
}}}}
\put(1651,-2086){\makebox(0,0)[lb]{\smash{{\SetFigFont{12}{14.4}{\rmdefault}{\mddefault}{\updefault}{\color[rgb]{0,0,0}$(-1, 2)$}%
}}}}
\put(751,-2686){\makebox(0,0)[lb]{\smash{{\SetFigFont{12}{14.4}{\rmdefault}{\mddefault}{\updefault}{\color[rgb]{0,0,0}$(0, 1)$}%
}}}}
\end{picture}%